\documentclass[a4paper]{amsart}

\usepackage{amsfonts}
\usepackage{amsmath}
\usepackage{amsthm}
\usepackage{amssymb}
\usepackage{yhmath}
\usepackage{enumerate}

\usepackage{xcolor}

\theoremstyle{definition} \newtheorem{defn}{Definition}[section]
\theoremstyle{plain} \newtheorem{thm}[defn]{Theorem}
\theoremstyle{plain} \newtheorem{propn}[defn]{Proposition}
\theoremstyle{plain} \newtheorem{lemma}[defn]{Lemma}
\theoremstyle{plain} \newtheorem{cor}[defn]{Corollary}
\theoremstyle{plain} \newtheorem*{claim*}{Claim}
\theoremstyle{plain} \newtheorem{claim}{Claim}

\theoremstyle{definition}  \newtheorem*{ex}{Example}
\theoremstyle{definition}  
\theoremstyle{remark}  
\theoremstyle{remark}  
\theoremstyle{remark}  

\theoremstyle{plain} \newtheorem*{thm*}{Theorem}
\theoremstyle{plain} \newtheorem*{cor*}{Corollary}

\newcommand {\R} {\mathbb{R}}
\newcommand {\Q} {\mathbb{Q}}

\newcommand {\Rbar} {\overline{\mathbb{R}}}

\renewcommand {\epsilon}{\varepsilon}

\newcommand {\M} {\mathcal{M}}

\renewcommand {\iff} {\leftrightarrow}

\newcommand{\Rtilde} {\widetilde{\R}}

\DeclareMathOperator {\an}{an}

\begin{document}
\title{Powers are easy to avoid}

\author[Jones]{Gareth Jones}
\address{School of Mathematics,
The University of Manchester,
Oxford Rd,
Manchester,
M13 9PL,
UK}

\email{Gareth.Jones-3@manchester.ac.uk}
\urladdr{http://personalpages.manchester.ac.uk/staff/Gareth.Jones-3/index.php}

\author[Le Gal]{Olivier Le Gal}
\address{Univ. Grenoble Alpes, Univ. Savoie Mont Blanc, CNRS,
LAMA, 73000 Chamb\'ery, France}
\email{Olivier.Le-Gal@univ-smb.fr}
\urladdr{http://www.lama.univ-savoie.fr/~LeGal/}

\maketitle

\begin{abstract}Suppose that $\widetilde{\R}$ is an o-minimal expansion of the real field in which restricted power functions are definable. We show that if $\widehat{\R}$ is both a reduct (in the sense of definability) of the expansion $\widetilde{\R}^{\R}$ of $\widetilde{\R}$ by all real power functions and an expansion (again in the sense of definability) of $\widetilde{\R}$, then, provided that $\widetilde{\R}$ and $\widehat{\R}$ have the same field of exponents, they define the same sets. This can be viewed as a polynomially bounded version of an old conjecture of van den Dries and Miller.
	\end{abstract}
\section{Introduction}
This paper deals with first order definability and more precisely inter-definability in o-minimal expansions of the field of reals, in the 
spirit of Bianconi \cite{bia:expsin, bia:restrpower}, van den Dries, Macintyre and Marker \cite[section 5]{vdd:mac:mar:logexp}, or Jones, 
Kirby and Servi \cite{jks} (see also Sfouli \cite{sfo:sin}, Jones, Kirby, Le Gal and Servi \cite{jklgs}, Wilkie \cite{wil:wilconj} and \cite{ElSav} for considerations in the same vein, and Conant \cite{con:presburger} for a very similar result in a very different setting). We assume that the reader is familiar with the basics of o-minimal geometry, and otherwise refer to Coste \cite{cos:omin} or van den Dries \cite{vdd:tame}. The question we address here is the analogue of a conjecture of van den Dries and Miller \cite{vdd:mil:geocat} concerning the exponential function. Our analogue considers power functions in place of the exponential function.

We begin by discussing the original conjecture of van den Dries and Miller. For this we denote by $\mathbb R_{\text{an}}$ the expansion of 
the real field by restricted analytic functions, and by $\mathbb R_{\text{an}}^{\mathbb R}$ the expansion of $\mathbb R_{\text{an}}$ by 
all real power functions
$$x \mapsto \left\{\begin{array}{ll}x^r,& x>0\\ 0,& x\le 0,\end{array}\right.$$
for real $x$ and arbitrary $r$ in $\mathbb R$. The field of reals with real power functions was first studied by Miller in \cite{mil:power}, and the survey \cite{mil:fields} highlights the central role of those functions in o-minimality. In particular, they form a complete scale for asymptotics in polynomially bounded structures. Finally, by $\mathbb R_{\text{an},\exp}$ we  denote the expansion of $\mathbb R_{\text{an}}$ by the 
exponential function. The three structures $\mathbb R_{\text{an}}$, $\mathbb R_{\text{an}}^{\mathbb R}$ and $\mathbb R_{\text{an},\exp}$ are known to be o-minimal (for instance the first two are reducts of the third, which is o-minimal by \cite{vdd:mil:Ranexp}), and to be genuine expansions one of another, in the sense that new sets are definable. The latter can easily be seen at the level of one variable functions, as follows. The field of exponents of $\mathbb R_{\text{an}}$ is known to be $\mathbb Q$, meaning that $\mathbb R_{\text{an}}$ only defines the power functions that have rational exponents, while by construction, $\mathbb R_{\text{an}}^{\mathbb R}$ has field of exponents $\mathbb R$. Furthermore, $\mathbb R_{\text{an}}^{\mathbb R}$ is polynomially bounded (\cite{mil:power}), meaning that any function definable in  $\mathbb R_{\text{an}}^{\mathbb R}$ grows slower at infinity than a power function, whereas the exponential function is bounded by no power but is definable in $\mathbb R_{\text{an},\exp}$. In \cite[2.5(6),\,p.\,506]{vdd:mil:geocat} van den Dries and Miller conjectured that $\R_{\text{an}}^\R$  is maximal amongst polynomially bounded reducts of $\R_{\text{an},\exp}$. Let us clarify that here, and in general in this paper, reduct and expansion are meant in the sense of definability. So, $\mathcal R_1$ is a reduct of $\mathcal R_2$ and $\mathcal R_2$ is an expansion of $\mathcal R_1$ if $\mathcal R_2$ defines all sets definable in $\mathcal R_1$. Also, definability is always meant with parameters.

We give a rough interpretation of the conjecture. The structure $\mathbb R_{\text{an}}$ defines the restriction of the exponential function to any 
bounded interval, so definable sets of $\mathbb R_{\text{an},\exp}$ are obtained by mixing global subanalytic sets with the germ of the exponential at infinity. This germ produces at least two phenomena that are not of subanalytic nature: irrational powers (by $x^r=e^{r\log(x)}$) and infinitely flat functions (such as $e^{-\frac{1}{|x|}}$ at $0$). Flat functions define in return germs that are not polynomially bounded, and then the unrestricted exponential function is definable by Miller's dichotomy theorem (\cite{mil:exp}). Hence, if the conjecture were proven to be true, those irrational powers (and their mixing with global subanalytic sets) would be the only thing that the unbounded exponential adds to global subanalytic sets that does not, in return, define the unbounded exponential.
To the best of our knowledge, the conjecture remains open, despite the progress made by Soufflet \cite{sou:vddmilconj}, Kuhlmann and Kuhlmann \cite{kuh:kuh} and Randriambololona \cite{ran:vddmilconj}. We discuss their results later.

We can ask similar questions at the polynomially bounded level. Here, having a certain field as field of exponents is the analogue of being polynomially bounded, and adding all real exponents is the analogue of adding the exponential function. In this setting we give a complete answer to the analogue of the question asked by van den Dries and Miller. We introduce some notation before stating our result. Given an o-minimal structure $\Rtilde$ we write $k(\Rtilde)$ for the field of exponents of $\Rtilde$, and we write $\Rtilde^{\Lambda}$ for the expansion of $\Rtilde$ by the real power functions with exponents in $\Lambda$. We say that $\Rtilde$ defines restricted power functions if, for any $r\in\mathbb R$, the restriction to some neighborhood of $0$ of $x\mapsto (1+x)^r$ is definable in $\Rtilde$. Defining the restricted exponential function is sufficient to define restricted powers, so this condition is not rare. Following \cite{mil:fields}, if $\Rtilde$ is polynomially bounded and defines restricted power functions, $\Rtilde^\R$ is polynomially bounded again. (The question of whether $\Rtilde^\R$ is polynomially bounded assuming  {only} that $\Rtilde$ is polynomially bounded remains open.) For brevity, the following statement of our main result does not mention the assumption of being polynomially bounded. But the interest is in this case, since otherwise $\Rtilde$ and $\Rtilde^\R$ already define the same sets. 

\begin{thm}\label{main} Let $\Rtilde$ be an o-minimal expansion of the real field that defines restricted power functions, and let $\widehat{\R}$ be an expansion of $\Rtilde$ and a reduct of $\Rtilde^\R$ (all in the sense of definability). If $\widehat{\R}$ and $\widetilde{\R}$ have the same field of exponents, then $\widehat{\R}$ and $\widetilde{\R}$ have the same definable sets.
\end{thm}
So $\Rtilde$ is maximal amongst reducts of $\Rtilde^\R$ with field of exponents $k(\Rtilde)$. 

An alternative way to interpret Theorem \ref{main} explains the title of the paper. Suppose $\Rtilde$ defines restricted power functions, and that $X$ is definable in $\Rtilde^\R$. If the field of exponents of the expansion $\Rtilde_X$ of $\Rtilde$ by $X$ is a certain field $\Lambda$, then Theorem \ref{main} shows that $X$ is definable in the expansion $\Rtilde^{\Lambda}$ of $\Rtilde$ by power functions with exponents in $\Lambda$. For $\Rtilde_X$ is an expansion of $\Rtilde^{\Lambda}$ and a reduct of $\Rtilde^{\mathbb R}$, and it has the same field of exponents $\Lambda$ as $\Rtilde^{\Lambda}$ (one has $k(\Rtilde^{\Lambda})=\Lambda$ in view of \cite[Theorem 1.4]{mil:fields}). So, roughly, if $X$ is defined from $\Rtilde$ and powers, then $X$ admits a definition which avoids all those powers that are not in return defined by $X$ and $\Rtilde$. We illustrate this with an example.
\begin{ex}
We define $f:\mathbb R^2\to\mathbb R$ by
\begin{equation}\label{eq:ex}
f(x,y)=\left\{\begin{array}{ll}\displaystyle y^{\frac{1}{\pi}}g\left(\frac{y}{x^{\pi}}-1\right),& \text{if}\; 0<x<1\;\text{and}\; x^{\pi}<y<x^{\pi}(1+x),\\
x & \text{otherwise},\end{array}\right.
\end{equation}
where $g$ is an analytic function on $(-1,+\infty)$ that we reveal below. The function $f$ is definable in $\mathbb R_{\text{an}}^{\mathbb R}$, being a composition of algebraic functions, powers, and restricted analytic functions (only the restriction of $g$ to $(0,1)$ intervenes) on domains defined using the same functions. 
Notice that the given definition involves unrestricted irrational power functions, since the powers in $y^{\frac{1}{\pi}}$ and $x^{\pi}$ act on variables that  approach zero.
However, with our particular $g$, the field of exponents of the expansion $\mathbb{R}_{\text{an},f}$ of $\mathbb{R}_{\text{an}}$ by $f$ is $\mathbb Q$. So the expansion $\mathbb{R}_{\text{an},f}$ of $\mathbb{R}_{\text{an}}$ is a reduct of $\mathbb R_{\text{an}}^{\mathbb R}$ with the same field of exponents as $\mathbb{R}_{\text{an}}$. Then Theorem \ref{main} asserts that $f$ is definable in $\mathbb R_{\text{an}}$; in other terms, $f$ can be defined with no use of irrational powers.
As the reader may have guessed our function $g$ is given by $g(t)=(1+t)^{-\frac{1}{\pi}}$, so $f$ admits actually the (subanalytic) definition $f(x,y)=x$. In this example, irrational powers  were indeed easy to avoid. The expression (\ref{eq:ex}) of $f$ now seems artificial, but has been chosen in view of our proof, because \eqref{eq:ex} is an admissible output of the Preparation Theorem of van den Dries, Speissegger (\cite{vdd:spe:preparation}) applied to $f$ in the structure $\mathbb R_{\text{an}}^{\mathbb R}$.
\end{ex}

In the aforementioned \cite{ran:vddmilconj}, Randriambololona shows that there are infinitely many different maximal polynomially bounded reducts of 
$\mathbb R_{\text{an},\exp}$. Randriambololona's arguments for producing different maximal polynomially bounded reducts of 
$\R_{\text{an},\exp}$ make heavy use of the exponential function, and do not adapt to the polynomially bounded setting. In our setting, we do not know whether $\Rtilde$ is the unique maximal reduct of $\Rtilde^{\mathbb R}$ 
with field of exponents $k(\Rtilde)$. We leave it 
as an open question, that we find exiting because we suspect the answer might be negative. Randriambololona also mention that, if van den Dries and Miller's conjecture was true, then $\mathbb R_{\text{an}}^{\mathbb R}$ would be the unique maximal polynomially bounded reduct of $\R_{\text{an},\exp}$ that expands $\mathbb R_{\text{an}}$ (even, $\mathbb R_{\text{an}}^{\mathbb R}$ would be the unique polynomially bounded reduct of $\R_{\text{an},\exp}$ with field of exponents $\mathbb R$ that expands $\mathbb R_{\text{an}}$). Our theorem proves an analogue of this,  {with} $\Rtilde$ playing the role of $\mathbb R_{\text{an}}$, {the expansion} $\Rtilde^{\Lambda}$ the role of $\mathbb R_{\text{an}}^{\mathbb R}$ and  {the structure} $\Rtilde^{\mathbb R}$ the role of $\mathbb R_{\text{an},\exp}$. Indeed, if $\widehat{\R}$ is an expansion of $\Rtilde$ with $k(\widehat{\R})=\Lambda$, then $\widehat{\R}$ is an expansion of $\Rtilde^{\Lambda}$, so if $\widehat{\R}$ is also a reduct of $\Rtilde^{\mathbb R}$, our theorem applies then $\Rtilde^{\Lambda}$ and $\widehat{\R}$ define the same sets. So Theorem \ref{main} proves the following:
\begin{cor}
Suppose $\Rtilde$ is an o-minimal expansion of the real field that defines restricted power functions, and $\Lambda\subset\mathbb R$ is a field extension of $k(\widetilde{\mathbb R})$. Then $\Rtilde^{\Lambda}$ is the unique (and in particular, unique maximal) reduct of $\Rtilde^{\mathbb R}$ with field of exponents $\Lambda$ that expands $\Rtilde$.
\end{cor}

Theorem \ref{main} also extends partial results around van den Dries and Miller conjecture.
The main advances in proving this conjecture concern functions of one variable.
Thus, Soufflet shows in \cite{sou:vddmilconj} that any expansion of $\mathbb R_{\text{an}}^{\mathbb R}$ that is a polynomially bounded reduct of $\mathbb 
R_{\text{an},\exp}$ defines the same one variable functions as $\mathbb R_{\text{an}}^{\mathbb R}$ (and so the same subsets of $
\mathbb R^2$ by cell decomposition). So from the point of view of definable unary functions, $\mathbb R_{\text{an}}^{\mathbb R}$ is maximal. 
A similar result is obtained 
independently by Kuhlman and Kuhlman in \cite{kuh:kuh}, dealing with arbitrary
polynomially bounded expansions $\widetilde{\mathbb R}$ of the field of reals in which the restricted exponential function is definable, rather 
than with $\mathbb R_{\text{an}}$ -- van den Dries and Miller focus on $\mathbb R_{\text{an}}$ but the generalized conjecture seems also natural (in 1996, few o-minimal structures were known that are not reducts or expansions of $\mathbb R_{\text{an}}$; from Le Gal \cite{leg:generic}, there are however polynomially bounded structures that are even not compatible with $\mathbb R_{\text{an}}$).
Kuhlman and Kuhlman moreover work in an arbitrary model $M$ of the complete theory $\text{Th}(\widetilde{\mathbb R}_{\exp})$ of the 
expansion of $\widetilde{\mathbb R}$ by the exponential. Namely, they prove (with obvious notations)
that if $\widehat{\mathcal M}$ is a polynomially bounded structure on $M$ that is an expansion of 
$\widetilde{\mathcal M}^{\mathbb R}$ and reduct of $\widetilde{\mathcal M}_{\exp}$, then $\widehat{\mathcal M}$ and 
$\widetilde{\mathcal M}^{\mathbb R}$ have the same Hardy field. As shown by Randriambololona (\cite{ran:vddmilconj}), having the same Hardy field and having the same unary functions are not equivalent in general, whereas they are over the reals. So \cite{kuh:kuh} answers, at the level of unary functions, a generalization of van den Dries and Miller conjecture for polynomially bounded expansions of the real field in which the restricted exponential is definable.
It is also possible to formulate a generalization of the conjecture of van den Dries and Miller to smaller fields of exponents. Our 
theorem combines naturally with \cite{kuh:kuh} to solve this question at the level of unary functions. 
\begin{propn}\label{prop:exp}
Let $\Rtilde$ be an o-minimal structure that defines the restricted exponential function, and let $\widehat{\mathbb R}$ be a polynomially bounded reduct of $\widetilde{\mathbb R}_{\exp}$ that expands $\Rtilde$. If $k(\widehat{\mathbb R})=\Lambda$, then $\widehat{\mathbb R}$ and $\widetilde{\mathbb R}^{\Lambda}$ defines the same one variable functions (then have the same definable subsets of $\mathbb R^2$).
\end{propn}
So, at the level of one variable functions, $\Rtilde^{\Lambda}$ is maximal amongst reducts of $\Rtilde_{\exp}$ with field of exponents $\Lambda$. Similarly, our theorem means that the original conjecture of van den Dries and Miller implies the generalization with smaller fields of exponents. 

To complete this exposition we state a corollary of the preceding proposition, obtained with the field of reals with restricted exponential as $\widetilde{\mathbb R}$, the expansion of $\widetilde{\mathbb R}$ by $f$ as $\widehat{\mathbb R}$ and $\mathbb Q$ as $\Lambda$. This statement was the initial motivation for this work.

\begin{cor}\label{subanalyticinrexp} Suppose that $f:\R\to \R$ is definable in both $\R_{\text{an}}$ and $\R_{\exp}$. Then $f$ is definable in the expansion of the real field by the restricted exponential function.
\end{cor}

The text is organized to follow the structure of the proof of Theorem \ref{main}. 
In the first step, a model-theoretic trick reduces the problem to one variable functions, 
at the cost of working over  {non-archimedean} models of $\text{Th}(\widetilde{\mathbb R}^{\mathbb R})$. For notational purposes, given a real closed field $M$ underlying a model of this theory, we decorate $\mathcal M$ with the symbol decorating $\mathbb R$ in the corresponding expansion (this is explained formally in Section \ref{model-theoretic}). 
The model-theoretic reduction is given in Section \ref{model-theoretic}, and
the problem becomes to show that if $f:M\to M$ is definable in a reduct 
$\widehat{\mathcal M}$ of $\widetilde{\mathcal{M}}^{\mathbb R}$ with $k(\widetilde{\mathcal M})=k(\widehat{\mathcal M})$,
then $f$ is definable in $\widetilde{\mathcal M}$. We prove the latter in the same section, using the lemmas which occur later in the paper (and form the core of the proof).

Those lemmas use the technical but simple Lemma  \ref{lem:2}, based on the Preparation Theorem (\cite{vdd:spe:preparation}) of van den Dries and Speissegger. This technical lemma 
gives uniforms pivots, in the following sense. Given various $p_i$,  in different expressions involving powers of $x\mapsto |x-p_i|$, we can, piecewise,  
simultaneously rewrite the expressions in  a way that only depends on powers of $x\mapsto |x-p|$ with a uniform $p$. We recall the statement of the Preparation Theorem (on which our whole proof heavily relies) and prove 
Lemma \ref{lem:2} in Section \ref{preparation}.

To show that $f$ is definable in $\widetilde{\mathcal M}$, we first prove that $f$ has a particular form. Namely, it is piecewise given as the composition of a map definable in $\widetilde{\M}$ with real power functions. This is Lemma \ref{lemma1} in Section \ref{terms}. It is related to a result of Miller's \cite[Proposition 4.5]{mil:power}. 
We use the fact that, being definable in $\widetilde{\mathcal{M}}^{\mathbb R}$, the function $f$ is piecewise given by terms, and reason inductively on the complexity of the terms. The proof requires rewriting powers of $\widetilde{\mathcal M}$-definable maps, and this can be done after giving a monomial form to such expressions. This is Lemma \ref{lem:3} in Section \ref{monom}, related to Lion and Rolin \cite[Th\'eor\`eme 1 (together with \S1 and \S2 p. 862) and Th\'eor\`eme 3]{lio:rol} and Soufflet \cite[Theorem 2.3]{sou:vddmilconj}. The lemma pays particular attention to the exponents that appear in the process, which we use later. 

We then have to avoid the extra exponents involved in the composition of an $\widetilde{\mathcal M}$-definable map with powers, using the assumption that the resulting function is definable in $\widehat{\mathcal M}$. This is Lemma \ref{lem:4}, proven in Section \ref{avoid} from the Preparation Theorem again and arguments involving valuations, including the strong version of our lemma \ref{lem:3}. It completes the proof of Theorem \ref{main}. The version of Lemma \ref{lem:4} we prove is a bit more general than the one we use, since we think it could have some independent interest.

Finally, in Section \ref{co-def}, we restate a version of Kuhlmann and Kuhlmann's Theorem adapted to our framework, and combine it with Theorem \ref{main} to show Proposition \ref{prop:exp}.

\subsection*{Acknowledgements}
 {This work was done while the first author was visiting Chamb\'ery, in the days when such things were possible. He would like to thank the department there for being excellent hosts.}

\section{The model-theoretic setting and the proof}\label{model-theoretic}

Let $\Rtilde$ be a polynomially bounded o-minimal expansion of the real  {ordered} field $\Rbar$ with field of exponents $k(\Rtilde)$. Let $\widetilde{L}$ be the language of $\Rtilde$ and assume, without loss of generality, that $\widetilde{T}=\text{Th} (\Rtilde)$ has quantifier elimination and a universal axiomatisation. Suppose too that all restricted power functions are definable in $\Rtilde$. Let $\Rtilde^\R$ be the expansion of $\Rtilde$ by all power functions.  Suppose that $\widehat{\R}$ is a reduct (in the sense of definability) of $\Rtilde^\R$. To prove Theorem \ref{main}, we can assume without loss of generality that $\widehat\R$ is a syntactic expansion of $\Rtilde$. Recall that we write $k(\widehat{\R})$ for the field of exponents of $\widehat{\R}$. 

So we have to show that if $k(\widehat{\R})=k(\Rtilde)$ then $\widehat{\R}$ defines the same sets as $\widetilde{\R}$.


Let $\widetilde{L}^\R$ be the language of $\Rtilde^\R$ and $\widetilde{T}^\R$ be the theory of $\Rtilde^\R$. 

 {
	Note that it follows immediately from the o-minimality of the pfaffian closure of $\Rtilde$ (due to Speissegger \cite{patrick:pfaffclosure}) that the structure $\Rtilde^\R$ is o-minimal. }

We will frequently use the following result, which is due to Chris Miller.
\begin{thm}[Miller \cite{mil:power,mil:fields}]\label{th:mil:term} The structure $\Rtilde^\R$ is polynomially bounded. The theory $\widetilde{T}^\R$ has quantifier elimination and a universal axiomatization. Definable functions in $\widetilde{\R}^\R$ are given piecewise by $\widetilde{L}^\R$ terms.
\end{thm}

Any model $\widetilde{\M}^\R$ of $\widetilde{T}^\R$ has a reduct $\widehat{\M}$ obtained as $\widehat{\R}$ is obtained from $\widetilde{\R}$. More precisely, $\widehat\R$ is an expansion of $\Rtilde$ by some sets $X_\iota$, with each $X_\iota$ the realisation in $\Rtilde^\R$ of a formula $\phi_\iota$. We define $\widehat\M$ to be the expansion of $\widetilde \M$ by the interpretations of the formulas $\phi_\iota$ in $\widetilde\M^\R$. The structures $\widehat\M$ and $\widehat\R$ are elementarily equivalent. 

By a result of van den Dries's (\cite[Lemma 4.7]{vdd:densepairs}), to reach our aim it is sufficient to establish the following.

\begin{claim}\label{cl:unary} Suppose that $\widetilde{\M}^\R$ is an $\omega$-saturated model of $\widetilde{T}^\R$. If $f:M\to M$ is definable in $\widehat{\M}$ then $f$ is definable in $\widetilde{\M}$.
\end{claim}

To see that our theorem follows from the claim, suppose that $X\subseteq \R^n$ is definable in $\widehat{\R}$. Let $\widetilde{\M}^\R$ be an $\omega$-saturated model of $\widetilde{T}^\R$. Suppose that $X=\phi(\widehat{\R})$ for some $\widehat{L}$-formula $\phi$ with parameters from $\R$. By the claim, and Lemma 4.7 in \cite{vdd:densepairs} every set definable in $\widehat{\M}$ is definable in $\widetilde{\M}$. So there is an $\widetilde{L}$-formula $\psi$ and parameters $a$ in $M$ such that $\widehat{\M}\models \forall x (\phi(x) \iff \psi(x,a))$. So $\widehat{\M}\models \exists y \forall x (\phi(x) \iff \psi (x,y))$ and so $\widehat{\R}\models\exists y \forall x (\phi(x) \iff \psi (x,y))$, so $X$ is definable in $\widetilde{\R}$.


 {We now prove the Claim, using results from later sections.}
\begin{proof}[Proof of Claim \ref{cl:unary}]
From Theorem \ref{th:mil:term}, $f$ is piecewise given by terms in the language $\widetilde{L}^{\mathbb R}$ with parameters from $M$. By Lemma \ref{lemma1}, there is a partition of $M$ into open intervals $I_1,\ldots, I_r$ and a finite set of points such that for each $i$ there exist $\alpha \in \R^m$, an $\widetilde\M$-definable $F$ with domain a subset of $M^m$ and $\delta\in M$ (all depending on $i$) such that
\[
f(x)= F\left(|x-\delta|^\alpha\right)
\]
on $I_i$. But our $f$ is also definable in $\widehat \M$. So by Lemma \ref{lem:4} we can refine the partition (but keep the notation) and get that
\[
f(x)=G\left(|x-\eta|^{\beta}\right),
\]
on $I_i$, for some $G: M^n\to M$ definable in $\widetilde{\mathcal M}$, $\beta \in k(\widehat{\mathcal M})^n$, $\eta \in M$.
Since $k(\widehat{\mathcal M})=k(\widetilde{\mathcal M})$ the function $x\mapsto |x-\eta|^{\beta}$ is in fact definable in $\widetilde{\mathcal M}$. And then $f$ is definable in $\widetilde{\mathcal M}$.
\end{proof}
\section{The valuation, the Preparation Theorem and uniform pivots}\label{preparation}

In this section we recall some classical facts about the interaction of the natural valuation map, o-minimality and polynomial boundedness. We let $\mathcal M$ be an o-minimal expansion of a real closed field, with underlying set $M$.

The natural valuation $v:M^\times\to M^\times/U$ is the quotient mapping from the group of invertible elements $M^\times$ to its quotient by the subgroup $U$ of units 
$$
U:=\left\{x\in M:\; \text{there is some } n\in \mathbb N \text{ such that } \frac{1}{n}<|x|<n\right\}.
$$ The value group $M^\times/U$ is a totally ordered group  which we write additively, where the order is conventionally chosen in such a way that $|x|\le |y|\Rightarrow v(x)\ge v(y)$, and $v$ is extended to $M$ by setting $v(0)=\infty$.
 {The valuation is only useful if $M$ is non-archimedean, for the value group is trivial otherwise. (Note that the model $\widetilde{\mathcal{M}}^{\mathbb R}$ we introduced in Section 2, in reducing our theorem to the claim, is certainly non-archimedean since the type given by the sequence of formulas $x>n$, for natural numbers $n$, is realized in $M$, as $\widetilde{\M}^\R$ is $\omega$-saturated). }
	
Also, if $\mathcal M$ defines real power functions, it induces an external multiplication on the value group that turns it into an ordered vector space over $\mathbb R$, by setting $\lambda \cdot v(x) = v(|x|^{\lambda})$ for $\lambda\in\mathbb R$.

 {Combining o-minimality with the valuation gives rise to the following.}
\begin{propn}\label{prop:units}
 {Let $f:X\to M$ be definable in an o-minimal expansion  $\mathcal M$ of a non-archimedean real closed field $M$. Suppose that for all $x\in X$ there is an $n\in \mathbb N$ such that $1/n< f(x)<n$. Then there is some $n\in \mathbb N$ such that for all $x\in X$ we have $1/n<f(x)<n$.}

\end{propn} 
\begin{proof}
 {We show that $f$ is bounded by below by a positive rational. The same argument for $1/f$ gives the upper bound.  
The function $f$ is definable and positive, so $f(X)$ is a finite union of points and intervals contained in $(0,+\infty)$. In particular, $f$ admits a nonnegative infimum, $a$ say. We claim that $v(a)\le 0$. So suppose that $v(a)>0$. Then by assumption $\{ a\} \not \subseteq f(V)$, so there is some interval $(a,b) \subseteq f(V)$. But then, as $a$ is either $0$ or a positive infinitesimal, there is some infinitesimal $c \in (a,b)$ with $c \in f(V)$, which contradicts our assumption (note that the $M$ has infinitesimals as it is non-archimedean). So $v(a)\le 0$ and hence $a> 1/n$ for some $n\in\mathbb N$.}

\end{proof}

The valuation of a function definable in a polynomially bounded o-minimal structure has strong regularities, as shown for instance by the preparation theorem of van den Dries and Speissegger \cite{vdd:spe:preparation}. We recall its statement, as we use it frequently in the following. Suppose now that $\M$ is a polynomially bounded o-minimal structure, expanding a real closed field $M$, and suppose that $\M$ is an elementary extension of an expansion of the real field. In the statement below, definable means definable in $\M$ (and cells are cells definable in $\M$).
\begin{thm} Suppose that $f_1,\ldots,f_m:M^{n+1}\to M$ are definable. Then there is a cell decomposition $\mathcal{C}$ of $M^{n+1}$
such that for each cell $C\in \mathcal{C}$ there are $\lambda_1,\ldots,\lambda_m\in k(\M)$ and definable functions $\theta,a_1,\ldots,a_m:M^n\to M$ and definable functions $u_1,\ldots,u_m:M^{n+1}\to M$ such that the graph of $\theta$ is disjoint from $C$, and for all $i=1,\ldots, m $ and $(x,y) \in C$ we have 
\[
f_i(x,y)= |y  - \theta(x)|^{\lambda_i} a_i(x) u_i(x,y)
\]
and
\[
|u_i(x,y)-1|<\frac{1}{2}.
\]
\end{thm}

This is only stated in \cite{vdd:spe:preparation} in the case that the underlying field $M$ is the real numbers. But it is clear that the proof there establishes the result above. (More care is needed in polynomially bounded structures whose theories have no archimedean model, see \cite{tyne:thesis}.) 

For simplicity, we will say that a function $u$ is a unit over  {$X$ if $X$ is included in the domain of $u$ and there is a positive integer $n$ such that $\frac{1}{n}<|u(x)|<n$ for all $x\in X$.} So, in the statement of the Preparation Theorem, the $u_i$ are units over $C$.
 {By Proposition \ref{prop:units}, for a function that is definable in an o-minimal expansion of a non-archimedean real closed field, being a unit is equivalent to have values that are all units.}

As a corollary of the Preparation Theorem, we get the following lemma. It could also be proved by a direct construction, but  {the} Preparation Theorem shortens the proof considerably.
Suppose that we have polynomially bounded o-minimal structures $\M_1,\ldots,\M_m$, each expanding the real closed field $M$. In 
the lemma, we write $\Lambda_i=k(\M_i)$, so $\Lambda_i$ is the field of exponents of $\M_i$. We denote by $(\Lambda_i^n)^*$ the
 space of $n$-linear forms with coefficients in $\Lambda_i$, and agree to apply an $\ell\in (\Lambda_i^n)^*$ to a $n$-tuple of real values, so $(\Lambda_i^n)^*$ is not strictly the dual space of $\Lambda_i^n$, but the corresponding subspace of the dual of $\mathbb R^n$. If $\alpha=(\alpha_1,\dots,\alpha_n)
\in\mathbb R^n$ is an $n$-tuple  {and $x \in M$ is positive,} we write $x^{\alpha}$ for the $n$-tuple $(x^{\alpha_1},\dots,x^{\alpha_n})
$. We make no distinction between an $m$-tuple of $n$-tuples and an $mn$-tuple; in particular, for $\alpha=(\alpha_1,\dots,\alpha_n)$, we have
$x^{(1,\alpha)} =(x,x^{\alpha})=(x,x^{\alpha_1},\dots,x^{\alpha_n})$.
The important point in the following statement is that $p_j$ depends on the interval $(a_j,a_{j+1})$ but not on the function (indexed by $i$).

\begin{lemma}\label{lem:2}
Let $(a,b)$ be an interval in $M$, and $n$ a positive integer. Suppose that
\begin{itemize}
\item $\ell_i\in(\Lambda_i^n)^*$, $\delta_i\in M$, $c_i\in M$, $\alpha^{(i)}\in\mathbb R^{n-1}$, with $\alpha^{(i)}=0$ if $\M_i$ does not define restricted power functions,
\item $u_i:M^{n}\to M$ is an $\M_i$-definable function such that $x\mapsto u_i(|x-\delta_i|^{(1,\alpha^{(i)})})$ is a unit over $(a,b)$,
\end{itemize}
for $i=1,\ldots,m$. Then there exists a finite partition of $(a,b)$ made of points and intervals,
$$(a,b)=\bigcup_{j=0}^s (a_j,a_{j+1})\cup \bigcup_{j=1}^{s} \{a_j\}$$ 
such that, for each $j= 0,\ldots s$, there exist $p_j\in M$, and for each $(i,j)\in\{1,\dots,m\}\times \{1,\dots,s\}$, there exist
\begin{itemize}
\item $c_{i,j}\in M$, $\ell_{i,j}\in(\Lambda_i^n)^*$,
\item $u_{i,j}:M^{n}\to M$ an $\M_i$-definable function such that $x\mapsto u_{i,j}\left(|x-p_j|^{(1,\alpha^{(i)})}\right)$ is a unit
over $(a_j,a_{j+1})$,
\end{itemize}
such that for all  $j=1,\ldots, s $ and for all $x\in(a_j,a_{j+1})$, we have
$$\;c_i|x-\delta_i|^{\ell_i(1,\alpha^{(i)})}u_i\left(|x-\delta_i|^{(1,\alpha^{(i)})}\right)=c_{i,j}|x-p_j|^{\ell_{i,j}(1,\alpha^{(i)})}u_{i,j}\left(|x-p_j|^{(1,\alpha^{(i)})}\right).$$
\end{lemma}

\begin{proof}
The proof is elementary while inconvenient to state explicitly: we just apply the Preparation Theorem simultaneously to each function $x\mapsto |x-\delta_i|$ in the field structure $\overline{M}$ and replace all $|x-\delta_i|$ by their prepared expressions.

The Preparation Theorem gives (with $n=1$ and $f_i(x)= |x-\delta_i|$) a partition of $(a,b)$ into points and intervals
\[
 (a,b)=\bigcup_{j=0}^s (a_j,a_{j+1})\cup \bigcup_{j=1}^{s} \{a_j\}
\]
and for each $j\in\{1,\dots,s\}$, a pivot $p_j\in M$,  and for  $(i,j)\in\{1,\dots,m\}\times \{1,\dots,s\}$, a rational number $q_{i,j}\in\mathbb Q$, a semialgebraic function $u'_{i,j}:M\to M$ and a constant $c'_{i,j}\in M$, such that the following holds. For all $i,j\in\{1,\dots,m\}\times \{1,\dots,s\}$ and all $x\in (a_j,a_{j+1})$ we have
\[
|x-\delta_i|=c'_{i,j}|x-p_j|^{q_{i,j}}u_{i,j}'(|x-p_j|)
\]
and the function
\[ x\mapsto u_{i,j}'(|x-p_j|)
\]
is a unit on $(a_j,a_{j+1})$. 

Then for $x\in (a_j,a_{j+1})$ we have, 
\[
|x-\delta_i|^{\ell_i(1,\alpha^{(i)})} = (c'_{i,j})^{{\ell_i(1,\alpha^{(i)})}}|x-p_j|^{q_{i,j}{\ell_i(1,\alpha^{(i)})}}u_{i,j}'\left(|x-p_j|\right)^{{\ell_i(1,\alpha^{(i)})}},
\]
and
\[
 u_i\left(|x-\delta_i|^{(1,\alpha^{(i)})}\right) = u_i\left(\left( c'_{i,j}|x-p_j|^{q_{i,j}}u_{i,j}'\left(|x-p_j|\right)\right)^{(1,\alpha^{(i)})}\right).
\]

If  $\mathcal M_i$ defines restricted power functions, $u_{i,j}'(|x-p_j|)^{{\ell_i(1,\alpha^{(i)})}}$ is definable in 
$\mathcal M_i$, and otherwise, $\alpha^{(i)}=0$ so ${\ell_i(1,\alpha^{(i)})}\in \Lambda_i$, and $u_{i,j}'(|x-p_j|)^{{\ell_i(1,\alpha^{(i)})}}$ is again definable in 
$\mathcal M_i$.
Then since $q_{i,j}$ is rational, $u_i\left(\left( c'_{i,j}|x-p_j|^{q_{i,j}}u_{i,j}'\left(|x-p_j|\right)\right)^{(1,\alpha^{(i)})}\right)$ is an $\mathcal M_i$-definable function in terms of $|x-p_j|^{(1,\alpha^{(i)})}$.

Set $c_{i,j}=c_i(c'_{i,j})^{{\ell_i(1,\alpha^{(i)})}}\in M$, $\ell_{i,j}=q_{i,j}\ell_i\in(\Lambda^n)^*$ and define
\[
u_{i,j}(y)= u_{i,j}'(y_1)^{\ell_i(1,\alpha^{(i)})}u_i\left(z_{i,j,0}(y),z_{i,j,1}(y),\dots,z_{i,j,n}(y)
\right),
\]
where $y= (y_1,\ldots, y_n)$ and $z_{i,j,0}(y)=c'_{i,j}y_1^{q_{i,j}}u_{i,j}'(y)$, and for $k=1,\dots,n$, 
$$z_{i,j,k}(y)=\left(c'_{i,j}\right)^{\alpha^{(i)}_{k}}y_1^{q_{i,j}}u_{i,j}'(y)^{\alpha^{(i)}_{k}}.$$ 
Then $u_{i,j}$ is $\mathcal M_i$-definable and $x\mapsto u_{i,j}(|x-p|^{(1,\alpha)})$ is a unit over $(a_j,a_{j+1})$. Finally, with these definitions, for $x\in (a_j,a_{j+1})$, we have 
\[
\;c_i|x-\delta_i|^{\ell_i(1,\alpha^{(i)})}u_i\left(|x-\delta_i|^{(1,\alpha^{(i)})}\right)=c_{i,j}|x-p_j|^{\ell_{i,j}(1,\alpha^{(i)})}u_{i,j}\left(|x-p_j|^{(1,\alpha^{(i)})}\right)
\]
as we wanted.
\end{proof}

The previous lemma, and its proof, show how arguments involving cell decompositions, even in one variable, can easily produce rather heavy notation. In an attempt to alleviate this, we say that a property $P(x)$ holds piecewise on   {an interval $(a,b)$ in $M$} if there is a finite partition $(a,b)=\bigcup_{i=0}^n (a_i,a_{i+1})\cup \bigcup_{i=1}^n\{a_i\}$ of $(a,b)$ made of points and intervals, such that, for all $i=1,\dots, n$, and all $ x\in (a_i,a_{i+1})$,  we have $P(x)$. We omit $(a,b)$ from the notation in the case that $(a,b)=M$, so saying $P(x)$ holds piecewise means $P(x)$ holds piecewise on $M$. Whenever $P(x)$ is an existential formula, then of course the witnesses for the existentially quantified variables depend on the piece. For instance, the conclusion of the preceding lemma can now be stated as follows.

Then, piecewise on $(a,b)$, there exists $p\in M$, and for $ i=1,\dots,m$, there exist
$ d_{i}\in M,  k_{i}\in(\Lambda_i^n)^*,$ and $\M_i$-definable $v_{i}:M^{n}\to M,\;
\text{ such that }x \mapsto v_{i}\left(|x-p_j|^{(1,\alpha^{(i)})}\right)\;\text{is a unit, and }$
$$c_i|x-\delta_i|^{\ell_i(1,\alpha^{(i)})}u_i\left(|x-\delta_i|^{(1,\alpha^{(i)})}\right)=d_{i}|x-p|^{k_{i}(1,\alpha^{(i)})}v_{i}\left(|x-p|^{(1,\alpha^{(i)})}\right).
$$


\section{Terms are simple}\label{terms}

In this section, we prove Lemma \ref{lemma1}, { which gives an expression for the unary terms in $\widetilde{L}^{\mathbb R}$:} they are piecewise compositions of an $\widetilde{\mathcal M}$-definable function with some real powers. 
 {The assumptions we need in this section are as follows. 
We suppose that $\widetilde{\mathbb R}$ is a polynomially bounded o-minimal expansion of the ordered field of real numbers in which restricted power functions are definable, 
with language $\tilde{L}$ and complete theory $\widetilde{T}$, and that
$\widetilde{\mathcal M}$ is a non-archimedean model of $\widetilde{T}$
which admits an expansion $\widetilde{\mathcal M}^{\mathbb R}$ to a model of $\widetilde{T}^{\mathbb R}$.}

\begin{lemma}\label{lemma1} If $t:M\to M$ is a term in $\widetilde{L}^\R$ with parameters from $M$, then piecewise on $M$, there exist an $\widetilde{\M}$-definable function $F:M^n\to M$, an element $\delta$ of $M$ and a $n$-tuple $\alpha$ of real numbers such that 
\[
t(x)= F\left(\left| x-\delta\right|^\alpha\right).
\]

\end{lemma}
\begin{proof}
We prove this by induction on the complexity of terms. For the base case there is nothing to do.

First suppose that $t(x)= G(t_1(x),\ldots,t_m(x))$ where $t_1,\ldots,t_m $ are $\widetilde{L}^\R$ terms with parameters from $M$, of lower complexity than $t$, and $G$ is an $\widetilde{\M}$-definable function. Applying the induction hypothesis to $t_1,\ldots,t_m$ and taking a common refinement of the resulting partitions we obtain that, piecewise, for $j=1,\ldots, m$ we have $n_j$-ary $\widetilde{\M}$-definable functions $F_j$, elements $\delta_j $ of $M$ and $n_j$-tuples $\alpha^{(j)}$ of real numbers such that 
\[
t_j(x)= F_j\left(\left| x-\delta_j\right|^{\alpha^{(j)}}\right).
\] 
By Lemma \ref{lem:2} we can further partition, so that 
\[
\left| x-\delta_j\right|=c_j \left| x-\delta\right|^{q_j}u_j(|x-\delta|)
\]
 where $c_j \in M,q_j\in \Q$ and $u_j$ is semialgebraic and $x\mapsto u_j(|x-\delta|)$ is a unit over the interval $I$ under consideration. Define
\[
H_j(y_0, y_{j,1},\ldots,y_{j,n_j})=F_j\left( c_j^{\alpha^{(j)}_1} y_{j,1}^{q_j} u_j(y_0)^{\alpha^{(j)}_1},\ldots,  c_j^{\alpha^{(j)}_{n_j}} y_{j,n_j}^{q_j} u_j(y_0)^{\alpha^{(j)}_{n_j}}\right)
\]
We then have
\[
t_j(x)=H_j(|x-\delta|,|x-\delta|^{\alpha^{(j)}_1},\ldots,|x-\delta|^{\alpha^{(j)}_{n_j}}),
\]
for $x\in I$. Since restricted powers are definable in $\widetilde{\M}$ and the function $x\mapsto u_j(|x-\delta|)$ is a unit on $I$, the function $H_j$ is definable in  {$\widetilde{\M}$} on a domain containing 
\[
\left\{ \left( |x-\delta|,|x-\delta|^{\alpha^{(j)}_1},\ldots,|x-\delta|^{\alpha^{(j)}_{n_j}}\right):\; x \in I\right\}.
\]
So, setting $Y_j=(y_{j,1},\ldots,y_{j,n_j})$, the function 
\[
F\left(y_0,Y_1,\dots,Y_m)\right) = G\left( H_1 \left(y_0,Y_1\right),\ldots,H_m\left(y_0,Y_m\right)\right)
\]
is definable on a domain containing 
\[
\left\{ |x-\delta|^\beta : x \in I_i\right\}
\]
where 
\[
\beta= \left(1, \alpha^{(1)}_1,\ldots, \alpha^{(1)}_{n_1},\ldots,\alpha^{(m)}_1,\ldots,\alpha^{(m)}_{n_m}\right)
\]
and on $I_i$ we have 
\[
t(x)= F\left( |x-\delta|^\beta\right)
\]
as required. \\

Otherwise we have $t(x)= (t_1(x))^\beta$ for some $\widetilde{L}^\R$ term $t_1$ with parameters from $M$, of lower complexity than $t$ and some $\beta\in\mathbb R$. Applying the induction hypothesis to $t_1$ we get that, piecewise, 
\[
t_1(x)= F\left( |x-\delta|^\alpha\right)
\]
with $\delta \in M,\alpha \in \R^n$ and $F$ definable in $\widetilde{\M}$. Applying Lemma \ref{lem:3} we can assume that we have
\[
t_1(x)= c |x-\delta'|^{\ell(1,\alpha)}u \left( |x-\delta'|^\alpha\right)
\]
where $c,\delta' \in M$ with $\ell$ a $k(\widetilde{\M})$-linear form and $u$ an $\widetilde{\M}$ definable function such that $x\mapsto u \left( |x-\delta'|^\alpha\right)$ is a unit on the interval under consideration. 

Let 
\[
G(y_0,\ldots,y_k)= c^\beta y_0 u(y_1,\ldots,y_k)^\beta.
\]
Since $u$ is a unit, $G$ is $\widetilde{\mathcal M}$-definable. And we have
\[
G\left( \left| x-\delta'\right|^{\beta\cdot \ell (1,\alpha)}, \left| x-\delta'\right|^{\alpha_1},\ldots,  \left| x-\delta'\right|^{\alpha_k}\right)=t(x),
\]
as required.
\end{proof}

\section{Monomialization of simple terms}\label{monom}

This section is devoted to Lemma \ref{lem:3}, which was used in the preceding section.  {The lemma} gives a piecewise monomial expression for the composition of an $\widetilde{\mathcal M}$-definable function and power functions.  {As in the preceding section, our proof needs that $\widetilde{\mathcal M}$ is a non-archimedean model of $\widetilde{T}$, the complete theory of an o-minimal and polynomially bounded expansion $\widetilde{\mathbb R}$ of the field of real numbers in which restricted power functions are definable, and that $\widetilde{\mathcal M}$ admits an expansion $\widetilde{\mathcal M}^{\mathbb R}$ to a model of $\widetilde{T}^{\mathbb R}$.}

\begin{lemma}\label{lem:3}
Let $F:M^{k}\to M$ be a function definable in $\widetilde{\mathcal M}$, and suppose that $\alpha\in\mathbb R^{k-1}$, and $\delta \in M$. Set  $\Lambda=k(\widetilde{\mathcal M})$. Then, piecewise for $x\in M$, there exist 
$(c, p)\in M^2$, 
$\ell \in (\Lambda^{k})^*$,
and an $\widetilde{\mathcal M}$-definable function $u:M^{k}\to M$  such that 
\[
F\left(|x-\delta|^{(1,\alpha)}\right)=c|x-p|^{\ell(1,\alpha)}u\left(|x-p|^{(1,\alpha)}\right)
\]
and
\[
x\mapsto u\left(|x-p|^{(1,\alpha)}\right)\]
is a unit.
\end{lemma}

\begin{proof}

We proceed by induction on $k$. In the base case $k=1$, in which we have no exponents $\alpha$, the function $F(|x- {\delta}|)$ is definable in $\widetilde{\mathcal M}$ and applying the Preparation Theorem gives the result. So we suppose that the result holds for $k$ and prove it for $k+1$. 

We first note that we can assume that $1,\alpha_1,\ldots,\alpha_k$ are linearly independent over $\Lambda$. If not, we can can assume that $1,\alpha_1,\ldots,\alpha_{k'}$ form a basis for the vector space $\text{Vect}_{\Lambda}(1,\alpha)$, for some $k'<k$. Then let $\lambda_{i,j}\in\Lambda$, for $(i,j)\in\{k'+1,\dots,k\}\times\{0,\dots,k'\}$ be such that
$$\alpha_i=\lambda_{i,0}+\sum_{j=1}^{k'} \lambda_{i,j}\alpha_{j},$$
and put $\displaystyle m_{i}(y_0,\dots,y_{k'})=\prod_{j=0}^{k'} y_j^{\lambda_{i,j}}$. We have 
\[
F\left(|x-\delta|^{(1,\alpha)}\right)=H\left(|x-\delta|^{(1,\alpha_1,\ldots,\alpha_{k'})}\right)
\]
with $H(y_0,\dots,y_{k'})=F(y_0,\dots,y_{k'},m_{k'+1}(y_0,\dots,y_{k'}),\dots,m_{k}(y_0,\dots,y_{k'}))$
and the result follows from the inductive assumption. So we suppose $(1,\alpha)$ are linearly independent over $\Lambda$.

We apply the Preparation Theorem to the function $F$ with respect to its last variable. This gives a cell decomposition $\mathcal C$ of $M^{k+1}$ into $\widetilde{\mathcal M}$-definable cells such that $F$ is prepared over each cell of $\mathcal C$.
Since the function $x\mapsto |x-\delta|^{(1,\alpha)}$ and $F$ are both definable in $\widetilde{\mathcal M}^{\mathbb R}$, we get by o-minimality that piecewise for $x\in M$, there is a unique cell $C\in\mathcal C$ that contains $|x-\delta|^{(1,\alpha)}$. We fix a piece $I\subset M$. Denote by $C'$ the projection of $C$ onto the first $k$ coordinates.

The previous preparation gives $\theta: C'\to M$, $a:C'\to M$ and $u:C\to M$, all definable in $\widetilde{\mathcal M}$, and a power $\lambda\in \Lambda$ such that for all  $ (y,y_{k+1})\in C$ we have
\[
F(y,y_{k+1})=|y_{k+1}-\theta(y)|^{\lambda}a(y)u(y,y_{k+1})
\]
and $u$ is a unit.
We set $\alpha'=(\alpha_1,\dots,\alpha_{k-1})$ and apply the inductive hypothesis to both $a\left(|x-\delta|^{(1,\alpha')}\right)$ and $\theta\left(|x-\delta|^{(1,\alpha')}\right)$ and then uniformize the pivots of those monomializations and of $|x-{\delta}|$ using Lemma \ref{lem:2}. So piecewise on $I$, we have the following

\begin{equation}\label{eq:3:1}\begin{array}{ll}
|x-\delta| & =c_{\delta}|x-p|^q u_{\delta}(|x-p|), \\
a\left(|x-\delta|^{(1,\alpha')}\right)  & = c_a|x-p|^{\ell_a(1,\alpha')}u_a\left(|x-p|^{(1,\alpha')}\right), \\
\theta\left(|x-\delta|^{(1,\alpha')}\right) & = c_{\theta}|x-p|^{\ell_{\theta}(1,\alpha')}u_{\theta}\left(|x-p|^{(1,\alpha')}\right),
\end{array}\end{equation}
with $p, c_{\delta}, c_a, c_{\theta}$ in $M$, 
$q\in\mathbb Q$, 
$\Lambda$-linear forms $\ell_a,\ell_{\theta}$ , an
 $\overline{M}$-definable function $u_{\delta}$, and 
 $\widetilde{\mathcal M}$-definable functions $u_a, u_{\theta}$, with
 $x\mapsto u_{\delta}(|x-p|), x\mapsto u_a(|x-p|^{(1,\alpha')})$ and $ x\mapsto u_{\theta}(|x-p|^{(1,\alpha')})$ units. 
Shrinking $I$ we may assume that it is one of the resulting pieces. 

By o-minimality, piecewise on $I$ again, exactly one of the following alternatives holds:
\begin{enumerate}[(a)]
\item $|x-\delta|^{\alpha_k} > 2 \left|\theta\left(|x-\delta|^{(1,\alpha')}\right)\right|$;
\item  $|x-\delta|^{\alpha_k} < \frac{1}{2} \left|\theta\left(|x-\delta|^{(1,\alpha')}\right)\right|$;
\item $\frac{1}{2} \left|\theta\left(|x-\delta|^{(1,\alpha')}\right)\right|<|x-\delta|^{\alpha_k} <2 \left|\theta\left(|x-\delta|^{(1,\alpha')}\right)\right|.$
\end{enumerate}
Shrinking $I$ again, we may assume that it is one of the resulting pieces, so that exactly one of these cases holds on $I$. We consider the three possibilities in turn. 

First suppose that (a) holds on $I$. We have
\begin{equation}\label{eq:3:2}\begin{array}{rl}
F\left(|x-\delta|^{(1,\alpha)}\right) &= \left||x-\delta|^{\alpha_k}-\theta\left(|x-\delta|^{(1,\alpha')}\right)\right|^{\lambda}a\left(|x-\delta|^{(1,\alpha')}\right)u\left(|x-\delta|^{(1,\alpha)}\right)\\
& =  |x-\delta|^{\lambda \alpha_k}\left|1- \frac{\theta\left(|x-\delta|^{(1,\alpha')}\right)}{|x-\delta|^{\alpha_k} }\right|^{\lambda}a\left(|x-\delta|^{(1,\alpha')}\right)u\left(|x-\delta|^{(1,\alpha)}\right).
\end{array}
\end{equation}
Here, because (a) holds on $I$, the term 
$$
1- \frac{\theta\left(|x-\delta|^{(1,\alpha')}\right)}{|x-\delta|^{\alpha_k}}
$$
is a unit on $I$, and thanks to (\ref{eq:3:1}) can be written as an $\widetilde{\mathcal M}$-definable function of $|x-p|^{(1,\alpha)}$. Then, replacing $a\left(|x-\delta|^{(1,\alpha')}\right)$ and $|x-\delta|$ in (\ref{eq:3:2}) by their expressions given in $(\ref{eq:3:1})$, and aggregating constants, units, and  powers of $|x-p|$, we get 
\[
F\left(|x-\delta|^{(1,\alpha)}\right) = c |x-p|^{q \lambda \alpha_k +\ell_a(1,\alpha')}u'\left(|x-p|^{(1,\alpha)}\right)
\]
with $c\in M$,  and $u'$ an  $\widetilde{\mathcal M}$-definable function such that $u'\left(|x-p|^{(1,\alpha)}\right)$ is a unit over $I$.
Setting $\ell(1,\alpha)= q \lambda \alpha_k +\ell_a(1,\alpha')$ gives the desired shape for $F\left(|x-\delta|^{(1,\alpha)}\right)$ on $I$.

The case (b) is very similar, factoring out $\theta\left(|x-\delta|^{(1,\alpha')}\right)$ instead of $|x-\delta|^{\alpha_k}$ in (\ref{eq:3:2}). We leave the calculation to the reader. 

It remains to consider the case that (c) holds on $I$. Notice that
(c) implies that the valuations of $\theta\left(|x-\delta|^{(1,\alpha')}\right)$ and of $|x-\delta|^{\alpha_k}$ are equal on $I$. From the expressions of these two terms given in (\ref{eq:3:1}) we then deduce
\[
\alpha_k v(c_{\delta}) + q\alpha_k v(|x-p|)= v(c_{\theta})+\ell_{\theta}(1,\alpha')v(|x-p|).
\]
We can rewrite this as 
\[
v(|x-p|)=\frac{v(c_{\theta})-\alpha_kv(c_{\delta})}{q\alpha_k-\ell_{\theta}(1,\alpha')}.
\]
The fact that $1,\alpha_1,\ldots,\alpha_{k}$ are linearly independent over $\Lambda$ implies that the denominator here does not vanish (or $q=0$ and $\ell_{\theta}=0$, but $q=0$ leads to the same ``constant times unit'' form for $|x-\delta|^{(1,\alpha)}$ we end up with below). We define an element $c$ of $M$ by
\[
c:=\left(\frac{c_{\theta}}{c_{\delta}^{\alpha_k}} \right)^{\frac{1}{q\alpha_k-\ell_{\theta}(1,\alpha')}},
\]
so that $v(|x-p|)=v(c)$ on $ I$. Thanks to Proposition \ref{prop:units}, this implies that there is an $\overline{M}$-definable unit $u_p$ such that $|x-p|=cu_p(|x-p|)$.
Now 
$$\begin{array}{rcl}
|x-\delta|^{(1,\alpha)}&=&\left(c_{\delta}|x-p|^q u_{\delta}\left(|x-p|\right)\right)^{(1,\alpha)}\\
&=&\left(c_{\delta}c^qu_p\left(|x-p|\right)^q u_{\delta}\left(|x-p|\right)\right)^{(1,\alpha)}.
\end{array}
$$
Since $\widetilde{\mathcal M}$ defines restricted power functions, 
the powers-of-units $\left(u_p(|x-p|)^q\right)^{(1,\alpha)}$ and $u_{\delta}(|x-p|)^{(1,\alpha)}$ 
are $\widetilde{\mathcal M}$-definable maps on  $I$. 
In particular, incorporating the constants and units in $F$, 
we get
\[
F\left(|x-\delta|^{(1,\alpha)}\right)=G(|x-p|)
\]
for $x \in I$, with some $\widetilde{\mathcal M}$-definable function $G$, which completes the proof. 
\end{proof}

\section{Avoiding extra powers}\label{avoid}

In this section we prove Lemma \ref{lem:4}. It is the core of the proof of Theorem \ref{main}, where extra powers are avoided.
 {We state the assumptions we need: 
$\widetilde{\mathcal M}$ is a non-archimedean model of $\widetilde{T}$, the complete theory of an o-minimal and polynomially bounded expansion $\widetilde{\mathbb R}$ of the field of real numbers that defines restricted power functions, and 
$\widetilde{\mathcal M}$ admits an expansion $\widetilde{\mathcal M}^{\mathbb R}$ to a model of $\widetilde{T}^{\mathbb R}$. We also fix some 
$\widehat{\mathcal M}$ a polynomially bounded expansion of $\widetilde{\mathcal M}$ whose theory admits an archimedean model. Our theorem only needs the particular case where $\widehat{\mathcal M}$ is a reduct of $\widetilde{\mathcal M}^{\mathbb R}$, but this hypothesis is not necessary for the proof of the lemma}.

\begin{lemma}\label{lem:4}
Let $F: M^n\to M$ be definable in $\widetilde{\mathcal M}$, $\delta\in M$,
$\alpha\in\mathbb R^n$ and $(a,b)$ be an interval in $M$.
Suppose that the function $x\mapsto F(|x-\delta|^{\alpha})$ is definable in $\widehat{\mathcal M}$ on $(a,b)$.
Then, piecewise on $(a,b)$, there exist $m\in \mathbb N, \eta \in M, \beta \in k(\widehat{M})^m$ and a function  {$G:M^m\to M$} definable in $\widetilde{ \M} $ such that 
\[ F(|x-\delta|^{\alpha}) = G(|x-\eta|^{\beta}).\]
\end{lemma}
\begin{proof}
Set $\Lambda =  k(\widehat{\mathcal M})$ and let $g$ be given by $g(|x-\delta|) =F(|x-\delta|^{\alpha})$ for $x\in (a,b)$.
We proceed by induction on $\text{Dim}(\text{Vect}_{\Lambda}(1,\alpha))$. If it is $1$, the powers in the expression $F(|x-\delta|^{\alpha})$ all belong to $\Lambda$ and the conclusion follows.

We suppose the lemma holds when $\text{Dim}(\text{Vect}_{\Lambda}(1,\alpha))=k$ and prove it for $k+1$.
Reordering the indices we can assume that $(1,\alpha')$ is a basis of $\text{Vect}_{\Lambda}(1,\alpha)$, where $\alpha'=(\alpha_1,\ldots,\alpha_{k})$. 
 For $(i,j)\in\{k+1,\dots,n\}\times\{0,\dots,k\}$ we let $\lambda_{i,j}$ be the unique elements of $\Lambda$ such that
$$\alpha_i=\lambda_{i,0}+\sum_{j=1}^{k} \lambda_{i,j}\alpha_{j}.$$
For $z=(z_0,\dots,z_k)$ and $i=k+1,\dots,n$, we define the monomial $m_{i}(z)=\prod_{j=0}^k z_j^{\lambda_{i,j}}$, and set $H(z)=F(z_1,\dots,z_k,m_{k+1}(z),\dots,m_{n}(z))$. We then have $g(|x-\delta|)=H(|x-\delta|^{(1,\alpha')})$, for $x\in(a,b)$. 

Since $H$ and $g$ are definable in $\widehat{\mathcal M}$, the set $S$ given by 
$$S:=\{(x,y)\in M\times M^k: g(x)=H(x,y)\}$$
is definable in $\widehat{\mathcal M}$. 
We apply cell decomposition to $S$, and get a  partition
$\mathcal C$ of $M\times M^k$ into cells definable in  $\widehat{\mathcal M}$ 
such that each cell is either included in $S$, or has empty intersection with $S$. 
For each cell $C\in\mathcal C$,  
since $\phi: x\mapsto |x-\delta|^{(1,\alpha')}$ is definable in $\widehat{\mathcal{M}}^{\mathbb R}$ the preimage $\phi^{-1}(C)$ is a finite union of points and intervals. 
Then, piecewise on $(a,b)$, the image of $\phi$ is contained in a unique cell $C$ of $\mathcal C$. We fix a piece $I$ and consider two cases according to the shape of $C$ with respect to the last variable $y_k$. Note that since the image of $\phi$ is included in $S$, and the cell decomposition $\mathcal C$ is compatible with $S$, the whole cell $C$ is included in $S$. 

First suppose that $C$ is a graph over a cell $D\subset M\times M^{k-1}$ so that
\[
C=\{(x,y',y_k)\in M\times M^{k-1}\times M:(x,y')\in D,\; y_k=h(x,y')\},
\]
with $h$ definable in $\widehat{\mathcal M}$. 
Then
\begin{equation}\label{lemma4eq1}
|x-\delta|^{\alpha_k}=h\left(|x-\delta|^{(1,\alpha'')}\right)
\end{equation}
for $x$ in $I$,  with $\alpha''=(\alpha_1,\dots,\alpha_{k-1})$. We monomialize $h\left(|x-\delta|^{(1,\alpha'')}\right)$ using Lemma \ref{lem:3}, so that piecewise on $I$, we have  {
\[
h\left(|x-\delta|^{\left(1,\alpha''\right)}\right)=c|x-p|^{\ell \left(1,\alpha''\right)}u\left(|x-p|^{(1,\alpha'')}\right)
\]}
for some constants $c,p$ in $M$, some $\Lambda$-linear form $\ell\in\left(\Lambda^{k}\right)^*$, and some unit $u$ definable in $\widehat{\mathcal M}$. We then uniformize the pivots of this monomialization and of $|x-\delta|$ thanks to Lemma \ref{lem:2} (but keep the notation for the monomialization), so piecewise on $I$, we have
\begin{equation}\label{lemma4eq2}
h\left(|x-\delta|^{(1,\alpha'')}\right)=c|x-p|^{\ell(1,\alpha'')}u\left(|x-p|^{(1,\alpha'')}\right)
\end{equation}
and
\begin{equation}\label{lemma4eq3}
 |x-\delta|=c'|x-p|^qu'(|x-p|),
\end{equation}
with $p,c,c'$ constants, $u$ unit definable in $\widehat{\mathcal M}$, $u'$ unit definable in $\overline{M}$, $\ell\in(\Lambda^{k})^*$ and $q\in\mathbb Q$.
We shrink $I$ to be one piece.

From equations \eqref{lemma4eq1},\eqref{lemma4eq2} and \eqref{lemma4eq3} together with the fact that $\alpha_k\notin \text{Vect}_{\Lambda}(1,\alpha')$ and so $q\alpha_k-\ell(1,\alpha'')\neq 0$, (or $q=0$ and $\ell=0$, but $q=0$ leads to the same ``constant times unit'' form for $|x-\delta|$ we end up with below), we deduce that 
\[
v\left(|x-p|\right)= \frac{v(c)-v(c'^{\alpha_k})}{q\alpha_k-\ell(1,\alpha'')}
\]
for all $x \in I$. In particular, setting $c''=\left(\frac{c}{c'^{\alpha_k}}\right)^{\frac{1}{q\alpha_k-\ell'(1,\alpha'')}}$, 
we have 
\[
v\left(|x-p|\right)=v(c'')
\] 
so, from Proposition \ref{prop:units}, there exists a function $u''$, definable in $\overline{M}$, such that 
\[
|x-p|=c''u''(|x-p|)
\] and such that 
\[
x\mapsto u''(|x-p|)
\]
 is a unit on $I$. We can conclude this case. We have
\begin{eqnarray*}
 |x-\delta|&=&c'(c''u''(|x-p|))^qu'(|x-p|) \\
 & =&c'''u'''(|x-p|)
\end{eqnarray*}
 with $c'''\in M$ and $x\mapsto u'''(|x-p|)$ a unit on $I$ definable in $\overline{ M}$. Since any real power of an $\overline M$-definable unit is definable in $\widetilde{\mathcal M}$, we get $g(|x-\delta|) = G(|x-p|)$, with $G$ definable in $\widetilde{\mathcal M}$, by simply replacing $|x-\delta|$ by $c'''u'''(|x-p|)$  {on the right of} the equality $g(|x-\delta|)=F(|x-\delta|^{\alpha}$).

We now treat the second case and suppose that $C$ is a strip between two graphs over a cell $D\subset M\times M^{k-1}$. So 
\[
C=\{(x,y',y_k)\in M\times M^{k-1}\times M:(x,y')\in D,\; h_1(x,y')<y_k<h_2(x,y')\}
\]
with $h_1<h_2$ definable in $\widehat{\mathcal M}$ (or with $h_1$ or $h_2$ equal to $-\infty$ or $+\infty$, respectively, with the obvious conventions). Recall that on $I$, the image of $\phi: x\mapsto |x-\delta|^{(1,\alpha')}$ is included in $C$, so 
we have
\begin{equation}\label{eq:4:1}
h_1\left(|x-\delta|^{(1,\alpha'')}\right)<|x-\delta|^{\alpha_k}< h_2\left(|x-\delta|^{(1,\alpha''})\right),
\end{equation}
with $\alpha''=(\alpha_1,\dots,\alpha_{k-1})$.
  We monomialize $h_1\left(|x-\delta|^{(1,\alpha'')}\right)$ and $h_2\left(|x-\delta|^{(1,\alpha'')}\right)$ with Lemma \ref{lem:3}, and uniformize the pivots of these monomializations and of $|x-\delta|$ with Lemma \ref{lem:2}. So, piecewise on $I$, we get:
\begin{eqnarray*}
h_1\left(|x-\delta|^{(1,\alpha'')}\right)&=&c_1|x-p|^{\ell_1(1,\alpha'')}u_1\left(|x-p|^{(1,\alpha'')}\right),\\
h_2\left(|x-\delta|^{(1,\alpha'')}\right)&=&c_2|x-p|^{\ell_2(1,\alpha'')}u_2\left(|x-p|^{(1,\alpha'')}\right),\\
 |x-\delta|&=&c|x-p|^qu(|x-p|),
 \end{eqnarray*}
with $c_1, c_2, c, p$ constants in $M$, a unit $u$ definable in $\overline{M}$, units $u_1, u_2$ definable in $\widehat{\mathcal M}$, forms $\ell_1,\ell_2$ in  $(\Lambda^{k})^*$ and $q\in\mathbb Q$. We now write $I$ for one of the resulting intervals. 

Note that $|x-\delta|^{\alpha_k}$ is positive. So we may, after possibly replacing $h_1$ by $0$, and constants and units by their negations, assume that $c_1,c_2,c$ are non-negative, that 
$u_1\left(|x-p|^{(1,\alpha'')}\right), u_2\left(|x-p|^{(1,\alpha'')}\right), u(|x-p|)$ are non-negative on $I$, and that the system of inequalities (\ref{eq:4:1}) holds. Taking the valuation of these inequalities and using the monomializations above gives
\begin{equation}\label{eq:4:2}
v(c_1)+\ell_1(1,\alpha'')v(|x-p|)\ge \alpha_k v(|x-p|) \ge v(c_2)+\ell_2(1,\alpha'')v(|x-p|)
\end{equation}
for $x \in I$.

We first consider the case that $v(|x-p|)$ is constant as $x$ ranges over $I$. So suppose $c_3$ in $M\setminus \{ 0\}$ is such that $v(|x-p|) = v(c_3)$ for all $x $ in $I$. By Proposition \ref{prop:units}, $|x-p|=c_3 u_3(|x-p|)$ for some $u_3$ definable in $\overline{ M}$ such that $u_3(|x-p|)$ is a unit.  Then  $|x-\delta|^{\alpha}= (cc_3^q u_3(|x-p|)^qu(|x-p|))^{\alpha} $ is a function of $|x-p|$ definable in $\widetilde{\mathcal M}$, and replacing it in $F\left(|x-\delta|^{\alpha}\right)$ gives an expression of the form $F\left(|x-\delta|^{\alpha}\right)=G(|x-p|)$ with $G$ definable in $\widetilde{\mathcal M}$, which finishes the proof in this case.

So we can assume that $v(|x-p|)$ is not constant for $x\in I$. Since the image of $I$ under $|x-p|$ is an interval and the image of an interval under the valuation $v$ is an interval in the value group, the image of $I$ under $x\mapsto v(|x-p|)$ is an interval, with endpoints $v_1 < v_2$ in the value group ($v_1$ and $v_2$ are distinct since $v(|x-p|)$ is not constant; $v_2$ might be $\infty$). 
In (\ref{eq:4:2}), since the slope $\alpha_k$ of the middle affine map $v(|x-p|)\mapsto \alpha_k v(|x-p|)$ does not belong to $\text{Vect}_{\Lambda}(1,\alpha'')$ while the slopes $\ell_1(1,\alpha'')$ and $\ell_2(1,\alpha'')$ of the outer maps do, the inequalities (\ref{eq:4:2}) are necessarily strict over the preimage $v^{-1}((v_1,v_2))$ of the open interval $(v_1,v_2)$ by $v$.
So the mean value of the bounds also satisfies the strict inequalities, that is, for all $x\in I$, if $|x-p|\in v^{-1}((v_1,v_2))$, we have
$$
v(c_1)+\ell_1(1,\alpha'')v(|x-p|) > v(a)+\ell(1,\alpha'')v(|x-p|) > v(c_2)+\ell_2(1,\alpha'')v(|x-p|),
$$
with $a=(c_1c_2)^{\frac{1}{2}}$ and $\ell=\frac{\ell_1+\ell_2}{2}$. (For the degenerate cases, 
if $h_1=0$ and $h_2\neq +\infty$, we choose $a$ with $v(a)>v(c_2)$ and set $\ell=\ell_2$;
if $h_1\neq 0$ and $h_2=+\infty$ we choose $a$ with $v(a)<v(c_2)$ and set $\ell=\ell_1$;
if $h_1=0$ and $h_2=+\infty$ we set $a=1$ and $\ell=0$).

From this we deduce that
\begin{equation}\label{eq:4:3}
 h_1\left(|x-\delta|^{(1,\alpha'')}\right) < a|x-p|^{\ell(1,\alpha'')}<h_2\left(|x-\delta|^{(1,\alpha'')}\right)
\end{equation}
for all $x\in I$ with $|x-p| \in v^{-1}((v_1,v_2))$. Since this system of inequalities is definable in $\widehat{\mathcal{M}}^\mathbb R$ we can partition $I$ into intervals on which (\ref{eq:4:3}) holds or  on which it fails. On the intervals on which \eqref{eq:4:3} fails, the image by $v$ of $|x-p|$ has to be $v_1$ or $v_2$, so $v(|x-p|)$ is constant, a case that has been treated already. So we may suppose, after shrinking $I$ again, that the inequalities (\ref{eq:4:3}) hold for $x\in I$.

These inequalities implies that the image of $I$ under $$x\mapsto \left(|x-\delta|^{(1,\alpha'')},a|x-p|^{\ell(1,\alpha'')}\right)$$ is included in the cell $C$, and so in  $S$. So we have
\[
 g(|x-\delta|)=H\left(|x-\delta|^{(1,\alpha'')},a|x-p|^{\ell(1,\alpha'')}\right),
\]
for $x$ in $I$. This means that, replacing  $|x-\delta|$ by $c|x-p|^qu(|x-p|)$, we have 
\[
g(|x-\delta|)=G\left(|x-p|^{\beta'}\right),
\]
with $G$ definable in $\widetilde{\mathcal M}$, and $\beta'=(q,q\alpha_1,\dots,q\alpha_{k-1},\ell(1,\alpha''))$.
Notice that $\ell\in (\Lambda^k)^*$ and $q\in\mathbb Q$, so 
\[
\text{Vect}_{\Lambda}\left(q,q\alpha_1,\dots,q\alpha_{k-1},\ell(1,\alpha'')\right)=\text{Vect}_{\Lambda}\left(1,\alpha''\right).
\]
Applying the inductive hypothesis to $G(|x-p|^{\beta'})$ over $I$ then gives the result.

\end{proof}

\section{Polynomially bounded unary functions in $(\widetilde{\mathbb R},\exp)$}\label{co-def}

In this section we prove Proposition \ref{prop:exp}. For this we fix a polynomially bounded structure $\Rtilde$ that defines the restricted exponential function, and a unary function $f:\R \to \R$ definable in $(\Rtilde, \exp)$.  It is sufficient to prove the next proposition.

\begin{propn}\label{ranandrexpgeneral} With the setting above, if $(\Rtilde, f)$ is polynomially bounded with field of exponents $\Lambda$, then $f$ is definable in $\Rtilde^{\Lambda}$. 
\end{propn}
To reduce to our main result, we need the following.
\begin{thm}[Kuhlmann and Kuhlmann] Suppose that $\wideparen{\R}$ is a polynomially bounded o-minimal structure which defines the restricted exponential function and has field of exponents $\R$. If $\widehat{\mathbb R}$ is a polynomially bounded reduct of $(\wideparen{\R},\exp)$ expanding $\wideparen{\R}$ then $\wideparen{\R}$ and $\widehat{\mathbb R}$ have the same Hardy field.
\end{thm}
In the case that $\wideparen{\R}$ is $\R_{\an}^\R$ this is a special case of the second part of \cite[Theorem 4 on p. 687]{kuh:kuh}. The proof of that Theorem works in the general setting above (for this, the hypothesis $k(\wideparen{\R})=\mathbb R$ is necessary; it is  implicit in the statement of \cite[Corollary 2]{kuh:kuh} as its proof involves \cite[Lemma 15]{kuh:kuh}).

\begin{proof}[Proof of Proposition \ref{ranandrexpgeneral}]
It is enough to prove that the germ of $f$ at each point (including $\pm \infty$) is definable in $\Rtilde^{\Lambda}$. We work at infinity. 
Denote by $\widehat{\mathbb R}$ the expansion $(\Rtilde, f, \{ x^\alpha: \alpha\in \R\})$ of $\widetilde{\mathbb{R}}$ by $f$ and power 
functions, and let $\mathcal{H}$ be the Hardy field of $\widehat{\mathbb R}$. The structure $\widehat{\mathbb R}$ is polynomially 
bounded by \cite[Theorem 1.4]{mil:fields} since the restricted exponential is definable in $(\Rtilde,f)$. Now $\widehat{\mathbb R}$ is 
a reduct of $(\widetilde{\mathbb R}^{\mathbb R},\exp)$ and 
an expansion of $\widetilde{\mathbb R}^{\mathbb R}$. So by the Theorem of Kuhlmann and Kuhlmann above (with $\wideparen{\mathbb R}=\widetilde{\mathbb R}^{\mathbb R}$), 
the field $\mathcal H$ is the Hardy field of $\widetilde{\mathbb R}^{\mathbb R}$. In particular, since the germ at infinity of $f$ 
belongs to $\mathcal H$, the restriction $f|_{(a,\infty)}$ is definable in $\Rtilde^{\mathbb R}$ for some real $a$. 

The structure $(\Rtilde, f|_{(a,\infty)}, \{ x^\alpha : \alpha \in \Lambda\})$ has field of exponents $\Lambda$ as a reduct of $(\widetilde{\mathbb R},f)$. As this structure is a reduct of $\Rtilde^{\mathbb R}$ and an expansion of $\Rtilde^{\Lambda}$ our main result applies, and so $f|_{(a,\infty)}$ is definable in $\Rtilde^{\Lambda}$. 
\end{proof}

\bibliographystyle{alpha}
\bibliography{bibli}

\end{document}